\documentclass[a4paper,11pt]{amsart}

\usepackage{amssymb}
\usepackage{mhequ}
\usepackage{url}

\usepackage{cite}

\DeclareFontFamily{OT1}{rsfs}{}
\DeclareFontShape{OT1}{rsfs}{m}{n}{ <-7> rsfs5 <7-10> rsfs7 <10->
rsfs10}{} \DeclareMathAlphabet{\mathscr}{OT1}{rsfs}{m}{n}
\newcommand{\mcM}{\mathscr M}

\newcommand{\eq}[1]{\eqref{#1}}

\newcommand{\bel}[1]{\begin{equation}\label{#1}}
\newcommand{\beal}[1]{\begin{eqnarray}\label{#1}}
\newcommand{\beadl}[1]{\begin{deqarr}\label{#1}}
\newcommand{\eeadl}[1]{\arrlabel{#1}\end{deqarr}}
\newcommand{\eeal}[1]{\label{#1}\end{eqnarray}}
\newcommand{\eead}[1]{\end{deqarr}}
\newcommand{\eea}{\end{eqnarray}}
\newcommand{\eeaa}{\end{eqnarray*}}

\newcommand{\be}{\begin{equation}}
\newcommand{\ee}{\end{equation}}

\DeclareFontFamily{OT1}{rsfs}{}
\DeclareFontShape{OT1}{rsfs}{m}{n}{ <-7> rsfs5 <7-10> rsfs7 <10->
rsfs10}{} \DeclareMathAlphabet{\mycal}{OT1}{rsfs}{m}{n}

\newcounter{mnotecount}[section]

\newcommand{\rmnote}[1]{}%{\mnote{#1}}

\newcommand{\Ric}{\operatorname{Ric}}

\def\mysavedown#1{\edef\mysubs{\mysubs#1}}
\def\mysaveup#1{\edef\mysups{\mysups#1}}
\def\mydown#1{{\mytensor}_{\vphantom{\mysubs}#1}}
\def\myup#1{{\mytensor}^{\vphantom{\mysups}#1}}
\def\tensor#1#2{
  #1
  \def\mytensor{\vphantom{#1}}
  \def\mysubs{\relax}
  \def\mysups{\relax}
  \let\down=\mysavedown
  \let\up=\mysaveup
  #2
  \let\down=\mydown
  \let\up=\myup
  #2
  }

\renewcommand{\sharp}{\#}

\newcommand{\Hess}{\operatorname{Hess}}

\newcommand{\Tr}{\operatorname{Tr}}

\newcommand{\II}{\mathbb I}

\newcommand{\A}{\mathbb A}
\newcommand{\T}{\mathbb T}
\newcommand{\R}{\mathbb R}

\renewcommand{\setminus}{\smallsetminus}

\renewcommand{\epsilon}{\varepsilon}
\renewcommand{\hat}{\widehat}

\def\crn#1#2{{\vcenter{\vbox{
        \hbox{\kern#2pt \vrule width.#2pt height#1pt
           }
          \hrule height.#2pt}}}}

\newcommand{\newF}{\lambda}

\newcommand{\grad}{\operatorname{grad}}

\renewcommand{\hbar}{{\overline h}}

\newcommand{\pre}[2]{{{\vphantom{#2}}^{#1}}\kern-.2ex{#2}}

\theoremstyle{plain}
\newtheorem{theorem}{\sc Theorem}[section]
\newtheorem{lemma}[theorem] {\sc Lemma}

\theoremstyle{definition}
\newtheorem{Remark}[theorem]{\sc  Remark\rm}
\newtheorem{remark}[theorem]{\sc  Remark\rm}

\numberwithin{equation}{section}

\date{January 31, 2011}

\begin{document}

\title[Unique continuation for stationary vacuum
space-times]{Unique continuation and extensions of Killing
vectors at boundaries for stationary vacuum space-times}
\author[P.T. Chru\'sciel]{Piotr T.~Chru\'sciel} \address{Piotr
T.~Chru\'sciel, University of Vienna, Austria}
\email{piotr.chrusciel@lmpt.univ-tours.fr} \urladdr{
http://www.phys.univ-tours.fr$/\sim${piotr}}
\author[E. Delay]{Erwann Delay} \address{Erwann Delay,
Laboratoire d'analyse non lin\'eaire et g\'eom\'etrie,
Facult\'e des Sciences, 33 rue Louis Pasteur, F84000 Avignon,
France} \email{Erwann.Delay@univ-avignon.fr}
\urladdr{http://math.univ-avignon.fr}
\begin{abstract}
Generalizing Riemannian theorems of Anderson-Herzlich and
Biquard, we show that two $(n+1)$-dimensional stationary vacuum
space-times (possibly with cosmological constant $\Lambda \in
\R$) that coincide up to order one along a timelike
hypersurface $\mycal T$ are isometric in a neighbourhood of
$\mycal T$. We further prove that KIDS of $\partial M$ extend
to Killing vectors near $\partial M$. In the AdS type setting, we show
unique continuation near conformal infinity if the metrics have
the same conformal infinity and the same undetermined term.
Extension near $\partial M$ of conformal Killing vectors of conformal
infinity which leave the undetermined Fefferman-Graham term
invariant is also established.
%We also discuss the local version
%of those results.
\end{abstract}

\maketitle

\tableofcontents

\section{Introduction}\label{section:intro}

Unique continuation theorems for initial data sets with Killing
Initial Data (KIDs) are of current interest (see,
e.g.,~\cite{IonescuKlainerman1,AIK2} and references therein).
Such theorems are relevant for uniqueness theorems for
stationary solutions~\cite{AIK,ACD2}. In a recent paper
\cite{BiquardContinuation} (see
also~\cite{AndersonHerzlichErratum}) unique continuation
theorems for the Riemannian Einstein equations have been
established. The aim of this work is to point out that
Biquard's arguments~\cite{BiquardContinuation} generalize in a
rather straightforward manner to initial data sets with a
timelike KID. More precisely, we prove that the stationary
Einstein equations, viewed as equations satisfied by the
initial data on a space-like surface, have the unique
continuation property at timelike hypersurfaces. While this is
hardly surprising, since stationary solutions are analytic in
the interior in an appropriate atlas~\cite{MzH}, it should be
borne in mind that the last reference does not apply to
solutions with boundary, as considered here. In particular the
solution need not to be analytic up to the boundary. See
Section~\ref{s18VI0.1} for further comments on this issue.

The precise result is as follows (compare
\cite[Theorem~4]{BiquardContinuation}). Consider  a space-time
$\mcM=\R\times M$, where $M$ is an $n$-dimensional manifold
with smooth compact
% \ptc{compactness only assumed on the boundary}
boundary $\partial M$; we denote by $t$ the coordinate along
the $\R$ factor. Let $g_-$ be a smooth Lorentzian  metric on
$\mcM$, with Killing vector $X=\partial/\partial t$. In adapted
coordinates such a metric can be written as
\beal{gme1}
 \phantom{xxx} g_{-} =
-V^2(dt+\underbrace{\theta_i dx^i}_{=\theta})^2 +
\underbrace{g_{ij}dx^i dx^j}_{=g_+}\;, &
%\\
 \mbox{with} &
 \partial_t V = \partial_t \theta = \partial_t g_+=0
 \;.
\eea%l{gme2}
We assume that the metric is Einstein,
\bel{EE} \Ric(g_{-})=
 \frac 2{n-1}\Lambda
 g_{-}
 \;,
\ee
where $\Lambda$ is a constant. In Section~\ref{sec:findist} we
prove:

\begin{theorem}
 \label{maintheorem3}
Let $n=\dim M\ge 2$,  and consider two $C^\infty$  stationary
Lorentzian Einstein metric of the form \eq{gme1}, with strictly
positive $V$ near $\partial M$, inducing the same metric on
$\partial M$. If the second fundamental forms of $\R\times
\partial M$ coincide,
the metrics are pull-backs of each other near $\R\times\partial
M$.
\end{theorem}

An infinitesimal version of Theorem~\ref{maintheorem3} is also
valid (compare~\cite[Theorem 2]{BiquardContinuation}):

\begin{theorem}
 \label{maintheorlinear}
Let $h$ be a smooth $t$--independent solution of the
linearization of the equation \eq{EE} at a stationary solution
as \eq{gme1}, defined near the boundary, and smooth up to the
boundary. Assume that $h$ has no $dx$ components%
\footnote{More precisely, $h$ as neither $dxdy$ nor $dydx$
components whatever $y=t,x$, or any coordinate on $\partial
M$.}
in a Gauss coordinate system near the boundary, where $x$ is
the distance to the boundary on $M$, and that $h=o(x^2)$.  Then
$h\equiv0$ near $\partial M$.
 %\ptc{c'est pres du bord, ou partout? probablement ajouter:
% "near $\partial M$". Ou alors: in the connected component of the set on which $\partial_t$ is timelike containing $\partial M$.
% But then a proof should be added, or at least hinted at. }
%%
\end{theorem}

In a manner  analogous to the Lorentzian
case~\cite{ChBeigKIDs},  couples $(\alpha,\nu)$ on $\partial M$
that satisfy equations \eqref{kids0}-\eqref{kids1} below (which
would be satisfied at $\partial M$ by normal and tangential
parts of a Killing form on $M$) will be called \emph{Killing
Initial Data} (KIDs) on $\partial M$. In the special case
$\alpha=0$, this reduces to the condition that $\nu^\sharp$
% \ptc{il est naturel de penser que $\nu$ est un champ de
% vecteur et non une forme dans la definition du KID, donc
% j'enleverais le $\sharp$ du $\nu^\sharp$}
 is
a Killing vector of the metric induced on $\partial M$, the
flow of which leaves the extrinsic curvature of $\partial M$
invariant. As a Corollary of Theorem~\ref{maintheorlinear}, in
Section~\ref{sub8VI0.1} we prove:

\begin{theorem}
 \label{maintheorem4}
Let $n=\dim M\ge 2$,  and let $g_-$ by a  $C^\infty$ stationary
Lorentzian Einstein metric of the form \eq{gme1}, with strictly
positive $V$ near $\partial M$. Then any time-independent KID
on $\R\times\partial M$ arises from the restriction to
$\R\times\partial M$ of a time-independent Killing vector
defined on a neighbourhood of $\R\times\partial M$.
\end{theorem}

There exist topological obstructions for global extensions,
compare Remark~\ref{R16VI0.1} below.

In the same vein, a  unique continuation result is established
for a stationary  metric $g_-$, satisfying the vacuum Einstein
equations with a negative cosmological constant and admitting a
$C^2$ conformal completion at infinity, with a smooth conformal
class at the conformal boundary. From~\cite[Theorem
7.1]{ChDelayStationary} the metric  $g_-$ is then
polyhomogeneous.

Then,  using a suitable coordinate  $\rho$  near $\partial M$
that vanishes at $\partial M$, the metric takes the form
$$
g_{-}=\rho^{-2}(d\rho^2+G(\rho)),
$$
where $G(\rho)$ is a family of Lorentzian stationary  metrics
on $\R\times\partial M$, which are Einstein and thus admit the
Fefferman-Graham expansions~\cite{FG,GrahamHirachi}%
\footnote{The original references assume a Riemannian signature
at the boundary, but the expansions are independent of this.}
:
$$
G(\rho)=
\left\{\begin{array}{ll}
G_0+G_2\rho^2+...+G_n\rho^n+...& n \mbox{ odd },\\
G_0+G_2\rho^2+...+{\mathcal O} \rho^n\ln\rho+G_n\rho^n+...& n \mbox{ even },
\end{array}\right.
$$
where all the coefficients $G_i$ are determined by  $G_0$, the
conformal infinity and the undetermined term $G_n$. In
particular if two metrics as above have the same $G_0$ and
$G_n$ in a common coordinate system, then
 they coincide to infinite
order.
%
%\ptc{added} The pair $(G_0,G_n)$ are defined only up to a
%conformal transformation of $G_0$, under which $G_n$ transforms
%in a suitable way.
% \ptc{this is probably complicated? the obstruction tensor is simple, but I am not
% sure about $G_n$}
% We denote by $[(G_0,G_n)]$ the corresponding equivalence
% class.

In the conformally compact setting our  result, proved in
Section~\ref{sec:ads}, reads:

\begin{theorem}\label{maintheorem1}
Let $n=\dim M\ge 2$,  and consider two $C^2$ compactifiable
stationary  Lorentzian Einstein metrics, with negative
cosmological constant, of the form \eq{gme1} which define the
same class $[(G_0,G_n)]$. Then the
metrics are diffeomorphic near infinity.
\end{theorem}

\begin{Remark}
 \label{R12VI0.1}
Theorem~\ref{maintheorem1} has also its infinitesimal version
similar to Theorem~2 in~\cite{BiquardContinuation}, which we do
not spell out in detail here.
\end{Remark}

Note that large families of stationary Lorentzian Einstein
metrics as above have been constructed in
\cite{ChDelayStationary}.

Consider, next, the associated problem of conformal isometries
extension from a conformal boundary at infinity.
 %\ptc{ici les arguments a la Obata ne s'appliquent pas, car on
% a une metrique Lorentzienne au bord et pas Riemannienne?}
Such boundary maps naturally decouple into conformal isometries
of the boundary which can be made into isometries by an
appropriate choice of the conformal factor, and those that
cannot. Here we only consider the former, and in
Section~\ref{sub8VI0.1} we prove (compare
\cite{WangConformal,ChDelayStationary,AndersonAdvMath2003} for
results under different conditions):

\begin{theorem}\label{maintheorem5a}
Let $n=\dim M\ge 2$,  and let $g$ be $C^2$ Lorentzian Einstein
metric of the form \eq{gme1} on $\R \times M $, with negative
cosmological constant,  and with a smooth conformal boundary at
infinity. Let $X$ be a conformal Killing vector of the
conformal boundary which is a Killing vector for some choice of
the conformal factor at the boundary. Then $X$ extends to a
Killing vector field defined near conformal infinity if and
only if $X$ leaves the associated undetermined term invariant.
\end{theorem}

The proofs are an adaptation of our context of a related
analysis of Biquard in the Riemannian
setting~\cite{BiquardContinuation}. Here one needs to control
some new terms related to the stationary Einstein equations,
which did not occur in Biquard's problem.

\section{Definitions, notations and conventions}\label{sec:def}

Let $\overline{M}$ be a smooth, compact $n$-dimensional
manifold with boundary $\partial {M}$, thus
$M:=\overline{M}\backslash\partial{M}$ is a non-compact
manifold without boundary. As already mentioned, we consider
both the case where the boundary is at finite distance, and the
case where the boundary $\partial {M}$ is a \emph{conformal
boundary at infinity}: we say that a Riemannian manifold
$(M,g)$ is {\it conformally compact} if there exists on
$\overline{M}$ a smooth defining function $\rho$ for $\partial
M$ (that is $\rho\in C^\infty(\overline{M})$, $\rho>0$ on $M$,
$\rho=0$ on $\partial {M}$ and $d\rho$ nowhere vanishing on
$\partial M$) such that $\overline{g}:=\rho^{2}g$ is a
$C^{2,\alpha}(\overline {M})\cap C^{\infty}_0(M)$ Riemannian
metric on $\overline{M}$. We will denote by $\hat{g}$ the
metric induced on $\partial M$. Our definitions of function
spaces follow~\cite{Lee:fredholm}. Now if
$|d\rho|_{\overline{g}}=1$ on $\partial M$, it is well known
(see~\cite{Mazzeo:hodge} for instance) that $g$ has
asymptotically sectional curvature $-1$ near its boundary at
infinity, in that case we say that $(M,g)$ is {\it
asymptotically hyperbolic}.

We recall that the  Lichnerowicz Laplacian acting on a
symmetric two-tensor field is defined as~\cite[\S~1.143]{Besse}
$$
\Delta_Lh_{ij}=-\nabla^k\nabla_kh_{ij}+R_{ik}h^k{_j}+R_{jk}h^k{_j}-2R_{ikjl}h^{kl}\;.
$$

The operator $\Delta_L$ arises naturally when  linearizing the
Ricci operator. Let us define the divergence of a covariant two
tensor $h$:
$$
(\delta h)_j:=-\nabla^ih_{ij},
$$
with symmetrized adjoint
$$
(\delta^*\omega)_{ij}:=\frac12(\nabla_i\omega_j+\nabla_j\omega_i).
$$
We use the following sign convention for the divergence of a
one form:
$$
d^*\omega:=-\nabla^i\omega_i.
$$

%
%the AdS
%space-time. In that case $M$ is the unit ball of $\R^n$, with the
%hyperbolic metric $$g_0=\rho^{-2}e\;,$$ $e$ is the Euclidean
%metric, $\rho(x)=\frac{1}{2}(1-|x|_e^2)$, and
%$$V_0=\rho^{-1}-1\;.$$

We denote by ${\mathcal T}^q_p$ the set of rank $p$ covariant and rank
$q$ contravariant tensors. When $p=2$ and $q=0$, we denote by
${\mathcal S}_2$ the subset of symmetric tensors. We use the summation
convention, indices are lowered and raised with $g_{ij}$ and its
inverse $g^{ij}$.

\section{Proof in the AdS type setting}
 \label{sec:ads}

In space-time dimension $n+1$ we consider a Lorentzian metric
$g_-$ of the form
$$
g_{-} = -V^2(dt+\theta)^2 + g_+=\rho^{-2}(d\rho^2+\widetilde G(\rho)),
$$
where $g_+$ is Riemannian. Thus  $\widetilde G(\rho)$ is a family of
Lorentzian metrics, parameterized by $\rho$, of the form
$$
\rho^{-2}\widetilde G(\rho)=-V^2(dt+\theta)^2+\rho^{-2}\widetilde g(\rho) \;.
$$
Here $\widetilde g(\rho)$ can be thought of as a family of Riemannian
metrics defined on the $(n-1)$-dimensional level sets of
$\rho$, and note that
$$
 g_+=\rho^{-2}(d\rho^2+\widetilde g(\rho))
 \;.
$$
%.
It is
convenient to introduce a coordinate $r=-\ln \rho$, and to
write $\rho^{-2}\widetilde G(\rho)=G(r)=G$ and
 $\rho^{-2}\widetilde g(\rho)=g(r)=g$.
Thus
$$
g_{-} =dr^2+G,
$$
where
$$
G=-V^2(dt+\theta)^2+g,
$$
so
$$
G^{-1}=(|\theta|^2-V^{-2})\partial_t^2-(\theta^\sharp\otimes\partial_t+\partial_t\otimes\theta^\sharp)+g^{-1}.
$$
The second fundamental forms of the level sets of $r$ are
$$
\II_{-}=\frac12  G'=-VV'(dt+\theta)^2-V^2\frac12[(dt+\theta)\otimes\theta'+\theta'\otimes(dt+\theta)]+\II,
$$
where primes denote partial $r$-derivatives and\footnote{The
reader is warned that our definition is the negative of that
in~\cite{BiquardContinuation}.}
$$\II=\frac12 g'.$$ Let also define the mean curvature
$$
H_-=\Tr_{G} \II_{-}=V^{-1}V'+H.
$$
Rescaling the metric to achieve a convenient normalization of
the constant $\Lambda=-n(n-1)/2$, the vacuum Einstein equations for a
metric satisfying \eq{gme1}-\eq{EE} read (see, e.g.,
\cite{Coquereaux:1988ne}) \bel{mainequation+tt}
 V(\nabla_{g_+}^*\partial V+n V)=\frac 1{4} |{}^+\newF |_g^2\;,
\ee
\bel{mainequation+ij}
     \Ric(g_+)+n g_+-V^{-1}\Hess_{g_+}V=\frac{1}{2V^{2}}{}^+\newF \circ {}^+\newF \;,
\ee
\bel{mainequation+ti} \delta_{g_+} (V {}^+\newF )=0\;,
 \ee
where
$$
 (^+\newF)_{ij}=-V^2(\partial_i \theta_j - \partial_j
\theta_i)\;,\;\;\;(^+\newF\circ
^+\newF)_{ij}=(^+\newF)_i{^k}(^+\newF)_{kj}\;.$$ As
$g_+=dr^2+g$, the non trivial Christoffel symbols of $g_+$ are
$$
^+\Gamma^r_{AB}=-\II_{AB}\;,\;\;^+\Gamma^A_{rB}=g^{AC}\II_{CB}\;,\;\;^+\Gamma^A_{BC}=\Gamma^A_{BC},
$$
in particular, the Hessian of $V$ is given by
$$
^+\Hess_{rr}V=V''\;,\;^+\Hess_{rA}V=d_AV'-g^{BC}\II_{AB}d_CV\;,
$$
$$
^+\Hess_{AB}V=(\Hess_{g})_{AB}V+V'\II_{AB}.
$$
Let us recall that $\theta$ is purely tangential:
$$
\theta=0 \, dr+\xi\;.
$$
\newcommand{\lambdap}{{}^+\lambda}%
\newcommand{\lambdam}{{}^_\lambda}%
Then $\lambdap =-V^2d_+\theta$ (here $d_+$ means the
differential on $M$), which gives
$$
 ^+\lambda_{rr}=0\;,\;\;^+\lambda_{rA}=-V^2\xi'_A
  \;, $$
$$^+\lambda_{AB}=-V^{2}(\partial_A\xi_B-\partial_B\xi_A)=
-V^2(d\xi)_{AB}=:\lambda_{AB}\;.
$$
So
$$
(\lambdap \circ \lambdap )_{rr}=V^{4}|\xi'|^2,
$$
$$
(\lambdap \circ \lambdap )_{rA}=-V^{2}(\xi'_C)g^{CD}(r)\lambda_{AD},
$$
$$
(\lambdap \circ \lambdap )_{AB}=
V^{4}(\xi'_A)(\xi'_B)+g^{CD}(r)\lambda_{AD}\lambda_{BC},
$$
$$
|\lambdap |^2=\Tr_{g}(\lambdap \circ \lambdap )=2V^{4}|\xi'|^2+|{\lambda}|^2,
$$
where ${\lambda}$ is the restriction of $\lambdap $ to the
level sets of $r$. Equation~\eq{mainequation+ij}  is equivalent
to
\begin{eqnarray}
 \label{mainequationAB}
 \lefteqn{
\Ric(g)-H\;\II-\II'+2\II\circ
\II+ng-V^{-1}\Hess_{g}V-V^{-1}V'\II
 }
 &&
\\
 \nonumber
 &&
  \phantom{xxxxxxxxxxxxxxxxxxxxx}  =\frac12
V^{2}(\xi')\otimes(\xi')+\frac12V^{-2}{\lambda}\circ
{\lambda}\;,
\end{eqnarray}
\bel{mainequationAr}
     -\delta_{g}\II-{d} H-V^{-1}[{d} V'-\II({\nabla}V,\cdot)]=\frac12\lambda(\cdot,(-\xi'))\;,
\ee \bel{mainequationrr}
-H'-|\II|^2+n-V^{-1}V''=\frac12V^{2}|\xi'|^2
 \;.
\ee
Equations~\eq{mainequationAr} and \eq{mainequationrr} can be
rewritten, respectively, as
\bel{mainequationAr2}
    \delta_{g}\II= -{d} H_--V^{-2}V'{d} V+V^{-1}\II({\nabla}V,\cdot)-\frac12\lambda(\cdot,-\xi')\;,
\ee
\bel{mainequationrr2}
H_-'=n-|\II|^2-(V^{-1}V')^2-\frac12V^{2}|\xi'|^2\;. \ee
The system \eq{mainequation+ti} is equivalent to
\bel{mainequationdivr} d^*(V^3\xi')=0\;, \ee
\bel{mainequationdivA} (V^3\xi')'+\delta
(V{\lambda})-HV^3\xi'-2V^2\II(\xi', \cdot)
 =0
 \;.
\ee
We want to show how Biquard's reduction of the Riemannian
vacuum Einstein equations to an elliptic system generalizes to
the problem at hand. From the linearization of the Ricci
curvature operator (see eg.~\cite{Besse}) we have
$$
 \Ric(g)'=\Delta_L (\II)-2\delta^*\delta \II-\delta^*dH
 \;,
$$
and so we obtain that  the $r$-derivative of equation
\eq{mainequationAB} reads
\begin{eqnarray}
 \label{Ricprime}
 \lefteqn{
 \Delta_L (\II)-\II''-2\delta^*\delta \II-\delta^*dH-V^{-1}\Hess V'
 }
 &&
\\
 &&
  +V^{-2}V'\Hess V+\mbox{ first order in } (\II,\xi',V')=0
 \;.
 \nonumber
\end{eqnarray}
Using \eq{mainequationAr2},  and the fact that
\begin{eqnarray*}
\delta^*dH_-&=&\delta^*dH+V^{-1}\Hess V'-V^{-2}V'\Hess V\\
 &&\hspace{1cm}+2V^{-3}V'dV\otimes dV-V^{-2}(dV\otimes dV'+dV'\otimes dV),
 \end{eqnarray*}
together with \eq{Ricprime}, one is led to an equation which is
elliptic for $\II$ if one disregards the fact that $H_-$ is
related to $\II$:
\bel{ellipticII} \nabla^*\nabla \II-\II''+ \mbox{ first order
in } (\II,\xi',V')=-\delta^*dH_-
 \;.
\ee
Next, the $r$-derivative of  \eq{mainequationdivA} reads
\bel{mainequationdivA'} V^3(\xi'''-\delta d\xi')+ \mbox{ first
order in } (\II,\xi',V')=0
 \;.
\ee
Inserting equality \eq{mainequationdivr} in the last equation
and dividing by $-V^3$ gives an elliptic equation for $\xi'$:
\bel{ellipticxi'}
 \nabla^*\nabla \xi'-(\xi')''+ \mbox{ first
 order in } (\II,\xi',V')=0
 \;.
\ee
Finally the $r$-derivative of  \eq{mainequation+tt} divided by
$V$ gives an  elliptic equation for $V'$:
\bel{ellipticV'}
 \nabla^*\nabla V'-(V')''+ \mbox{ first order in }
 (\II,\xi',V')=0
 \;.
\ee
Combining the three equations \eq{ellipticII}, \eq{ellipticxi'}
and \eq{ellipticV'}, and momentarily ignoring that $H_-$ is not
an independent field,  gives an  elliptic system for
$(V',\xi',\II)$  or, more geometrically, an elliptic system of
equations for
$$
\A:=VV'dt^2+\frac12V^2(\xi'\otimes dt+ dt\otimes \xi')+\II.
$$
Define now  the metric
$$
{\mathcal G}:=V^2dt^2+g.
$$
So  the ${\mathcal G}$-norm of $\A$ is:
$$
|\A|^2=|\II|^2+(V^{-1}V')^2+\frac12V^{2}|\xi'|^2,
$$
and the ${\mathcal G}$-trace of $\A$ is $H_-$. In particular,
\eq{mainequationrr2} becomes \bel{mainequationrr3}
H_-'=n-|\A|^2\;. \ee

Assume we have two stationary  Einstein metrics  $g_-$ and
$(g_0)_-$, and suppose that there exists a choice of conformal
factors at the boundary so that the metrics  coincide to order
$n$ at infinity in their respective coordinates $(\rho,x^A)$ as
above. Then they coincide  to infinite order.
 We wish to show
they are equal near infinity. We will put a subscript $0$ for
all quantities relative to $(g_0)_-$. First, we show that all
quantities relative to the difference between the metrics are
controlled by quantities relative to the difference of the
second fundamental form of the level set of $r$, and the same
is true for the difference between  the mean curvature.

As in~\cite{BiquardContinuation}, the   simplest control comes
from the fact that \bel{diffII} \II-\II_0=\frac12(g-g_0)'
 \;,
\ee
so from Equation (10) in~\cite{BiquardContinuation} for $s>2$,
\bel{contIIg} \int_{r_0}^\infty|\II-\II_0|^2_{g_0}e^{2sr}dr\geq
C^{-1}s^2\int_{r_0}^\infty|g-g_0|^2_{g_0}e^{2sr}dr
 \;.
\ee
Commuting \eqref{diffII} with derivatives, it is standard to
obtain
\begin{eqnarray}
 \label{contIIgb}
 \lefteqn{
  \int_{r_0}^\infty\sum_{i=0}^k|\nabla_0^{(i)}(\II-\II_0)|^2_{g_0}e^{2sr}dr
 }
 &&
\\
 &&
  \geq
 C^{-1}s^2\int_{r_0}^\infty\sum_{i=0}^k|\nabla_0^{(i)}(g-g_0)|^2_{g_0}e^{2sr}dr
  \;.
 \nonumber
\end{eqnarray}
\begin{remark}\label{remV}
Replacing  \eq{diffII} by  $(V^2)'-(V_0^2)'= [(V^2)-(V_0^2)]'$,
we will obtain the same kind of integral inequality comparing
$(V^2)'-(V_0^2)'$ with $(V^2)-(V_0^2)$.
\end{remark}

We also have (see \eq{mainequationrr3})
\bel{HmoinsH0} [H_--(H_0)_-]'=|\A_0|^2_{{\mathcal
G}_0}-|\A|^2_{\mathcal G}\;. \ee
Now recall that ${\mathcal G}$ and all of its derivatives are
bounded relatively to ${\mathcal G}_0$, so (see
Lemma~\ref{estiA}, Appendix~\ref{A12XI0.1} for details)
$$
 \left||\A_0|^2_{{\mathcal G}_0}-|\A|^2_{\mathcal G}
  \right|
 \leq C (|\A_0-\A|_{{\mathcal G}_0}+|{\mathcal G}_0-{\mathcal G}|_{{\mathcal G}_0}).
$$
Combining this with equation \eq{contIIgb}, Remark~\ref{remV}
and (derivatives of) equation \eq{HmoinsH0}, shows that
%
%\footnote{Here again, the metric $\mathcal G_0$  can be
%replaced by $dr^2+\mathcal G_0$.}
%
\begin{eqnarray}
 \label{contAg}
  \lefteqn{
 \int_{r_0}^\infty\sum_{i=0}^k|D_0^{(i)}(\A-\A_0)|^2_{\mathcal
 G^+_0}e^{2sr}dr
 }
 &&
\\
 &&
    \geq
C^{-1}s^2\int_{r_0}^\infty\sum_{i=0}^k|D_0^{(i)}[H_--(H_0)_-]|^2_{\mathcal
G^+_0}e^{2sr}dr
 \;,
 \nonumber
\end{eqnarray}
where $D_0$ is the covariant derivative
relative to $dr^2+\mathcal G_0$, $\A$ is identified with
$0dr^2+\A$ and the same for $\A_0$.

The rest of the proof is a straightforward  adaptation of
Section 3 of~\cite{BiquardContinuation} where $\II$ there
correspond to $\A$ here, $g$ there correspond to $dr^2+\mathcal
G$ here, etc. The argument there shows that $\A=\A_0$ so
$g_-=(g_0)_-$.  This concludes the proof of Theorem
\ref{maintheorem1}.

\section{Boundary at finite distance}
 \label{sec:findist}

We are interested now in stationary Lorentzian metrics (see
\eq{gme1}), solutions of
\bel{EElambda} \Ric(g_{-})=\lambda
g_{-}\;,\ee
where $\lambda$ is a constant, and where the $\{t=0\}$ slice is
a compact manifold $M$ with smooth boundary $\partial M$. All
quantities ($V$, $\theta$, $g_+$,...) are then assumed to be
smooth up to the boundary. If two metrics $g_-$ and $(g_0)_-$
as above coincide on the boundary together with their second
fundamental form, then (up to a diffeomorphism) they coincide
to order one and then to infinite order. We will show they are
in fact equal near the boundary. The proof proceeds exactly as
in Section \ref{sec:ads}, where we replace the parameter $r$ by
$x$, the distance to the boundary. Henceforth we write
$$
g_{-} = -V^2(dt+\theta)^2 + g_+=dx^2+G(x),
$$
thus
$$
G(x)=-V^2(dt+\theta)^2+g(x),
$$
where $g_+=dx^2+g(x)=dx^2+g$.  So we can write
$$
g_{-} =dx^2+G,
$$
where
$$
G=-V^2(dt+\theta)^2+g \;.
$$
Then, for  example,  \eq{contIIgb} is replaced by
\bel{contIIgx}
\int_{0}^{x_0}|\nabla_0^{(2)}(\II-\II_0)|^2_{g_0}x^{-2s}dx \geq
 C^{-1}s^2\int_{0}^{x_0}\sum_{i=0}^2s^{4-2i}x^{2i-4}|
 \nabla_0^{(i)}(g-g_0)|^2_{g_0}x^{-2s}dx
  \;,
\ee
with similar other obvious changes in the remainder of the
argument. This then gives the result for a boundary at finite
distance.

\section{KID extensions}
\subsection{The Riemannian case}
 \label{sub10VI0.1}

% \ptc{est-ce qu'on pourrait ameliorer le theoreme d'unicite? voir Qing; QIng says that stationary is static,
% but his claim that stationary implies static assumes that the metric at infinity is adS,
% which is wrong for e.g Kerr AdS, see eq. (1) dans~\cite{ChenLuPope}; but then our theorem says that there are
% metrics which have non trivial twist at infinity and spatially asymptote ads, so there isn't much that
% can be improved here, except perhaps the question of dimension}
%
We first give the result in the Riemannian setting, which does
not seem to have  appeared in the literature before. Let
$(M,g_+)$ be smooth $n$-dimensional Riemannian manifold with
smooth compact boundary $\partial M$. Take  $x$ to be the
geodesic distance from the boundary. Near $\partial M$ the
metric takes the form
\bel{10VI0.1}
 g_+=dx^2+g,
\ee
where $g=g(x)=g_{AB}dx^Adx^B$ is a family of metrics,
parameterized by $x$, on $\partial M$. Let us define
$\II=\frac12g'$, $H=Tr \II$ and $(\delta h)_j=-\nabla^ih_{ij}$
then, the non trivial Christoffel symbols are
$$
^+\Gamma^{x}_{AB}=-\II_{AB},\;\;^+\Gamma^{C}_{Ax}=\II^C_A,\;\;
^+\Gamma^C_{AB}=\Gamma(g)^C_{AB},
$$
in particular one has (see the Gauss and Codazzi equations in
\cite{Besse} for instance) \bel{ricpAB}
\Ric(g_+)_{AB}=(\Ric(g)-H\;\II-\II'+2\II\circ \II)_{AB}
 \;,
 \ee
\bel{ricpxA} \Ric(g_+)_{xA}=-\delta_{g}\II-{d} H
 \;,
 \ee
\bel{ricpxx} \Ric(g_+)_{xx}=-H'-|\II|^2
 \;.
\ee
Given $\omega=\alpha dx+\nu$, a one form decomposed in normal
and tangential parts, let $h=(\mathcal L_{\omega^\sharp}g_+)$,
then
\bel{hxx}
h_{xx}=2\alpha', \ee \bel{hxA}
h_{xA}=(\nu'_A-2\II^C_A\nu_C+d_A\alpha)=g_{AB}(\nu^B)'+d_A\alpha,
\ee \bel{hAB} h_{AB}=(\mathcal
L_{\nu^\sharp}g+2\alpha\II)_{AB}. \ee
We then see that if $\omega^\sharp$ is a Killing vector field of $M$, then
$\alpha'=0$.

We are interested in the Riemannian equivalent of \emph{Killing
initial data}, which we continue to call KIDs. By definition,
these are couples $(\alpha,\nu)$ on $\partial M$ that satisfy
equations which would be satisfied by normal and tangential
parts of a Killing form $\omega$ on $M$.

We first find necessary conditions on $(\alpha,\nu)$ on
$\partial M$ assuming $h\equiv 0$. Using $h_{AB}(0)=h'_
{AB}(0)=0$ (also using \eq{hxA} to replace $(\nu^\sharp)'$ and
\eq{ricpAB}, assuming $g_+$ is Einstein, to replace $\II'$ when
calculating $h'_{AB}(0)$),  we obtain the Riemannian KID
equations:
\bel{kids0} \mathcal
L_{\nu^\sharp}g+2\alpha\II=0, \ee
\bel{kids1} 2\mathcal
L_{\nu^\sharp}\II-\mathcal L_{\nabla
\alpha}g+2\alpha[\Ric(g)-\lambda-H\II+2\II\circ\II]=0. \ee

Reciprocally, assuming \eq{kids0} and \eq{kids1} on $\partial
M$, one can define $\alpha(x)=\alpha(0)$ and $\nu(x)$ such that
$h_{xA}=0$ in \eq{hxA}, that is $\nu^\sharp$ solves
$(\nu^\sharp)'=-\grad_g\alpha$. Then $h$ is purely tangential
(i.e., $h$ has no $dx$ components) and $h(0)=h'(0)=0$. Now as
$g_+$ is Einstein:
$$
\mbox{Ein}(g_+):=\Ric(g_+)-\lambda g_+=0,
$$
then  $D\mbox{Ein}(g_+)h=0$; this can be seen by a direct
calculation, or by   considering  the Einstein metric
$g_+^t=\Phi_t^*g_+$, where $\Phi_t$ is the local flow of
$\omega^\sharp$. From~\cite[Theorem 2]{BiquardContinuation} we
can conclude that $h\equiv 0$ near $\partial M$ so
$\omega^\sharp$ is a Killing near $\partial M$. We have thus
proved the following Riemannian equivalent of
Theorem~\ref{maintheorem4}:

\begin{theorem}
 \label{maintheorem5}
Let $n=\dim M\ge 2$,  and let $g_+$ by a  $C^\infty$   Einstein
metric on $M$. Then any  KID on $\partial M$ arises from the
restriction to $\partial M$ of a Killing vector defined  on a
neighborhood of $\partial M$.
\end{theorem}

\begin{Remark}
 \label{R16VI0.1}
The extensions above are not necessarily global. For example,
consider a sufficiently small ball  $B(p,\epsilon)$ in a flat
torus $\T^n$. If $
\partial B(p,\epsilon)$ is viewed as the boundary of the ball, then every
KID of the boundary extends to a globally defined Killing
vector in the interior. However, if  $
\partial B(p,\epsilon)$  is viewed as the
boundary of $\T^n\setminus B(p,\epsilon)$, then only those KIDs
which correspond to translations of the torus extend globally.
\end{Remark}

\subsection{The conformally compact Riemannian case}

We now treat the case where $(M,g_+)$ is conformally compact,
Einstein,  with a smooth conformal boundary at infinity. Let
$\rho$ be a defining function such that $|d\rho|_{{\overline
g}_+}=1$ near $\partial M$. We have
$$
\overline g_+=d\rho^2+\overline{g},
$$
where $\overline g=\overline g(\rho)=\overline g_{AB}dx^Adx^B$ is a family of metrics on $\partial M$.

Let us define $r=-\ln \rho$. Near $\partial M$, the metric
$g_+$ takes the form
$$
g_+=dr^2+g,
$$
where $g=g(r)=g_{AB}dx^Adx^B$ is a family of metrics on
$\partial M$. We then recover the form given for finite
distance boundary with $x$ replaced by $r$.

To establish the ``only if'' part of
Theorem~\ref{maintheorem5}, consider a Killing field
$\omega^\sharp$ for $g_+$. Denoting by $\nu^\sharp(0)$ the part
of $\omega^\sharp(0)$ tangent to $\partial M$, we further
suppose that $\nu^\sharp(0)$ is a
 Killing  vector field for $\overline
g(0)$. In the notation of ~\cite[Appendix A]{ACD2}, we then see
by the equations there that $\alpha\equiv 0$ and
$\nu^\sharp(x)=\nu^\sharp(0)$. We thus find the necessary
condition that
\bel{kidscc1} \mathcal
L_{\nu^\sharp(0)}\overline g_{(n-1)}=0, \ee
where $\overline g_{(n-1)}$ is the undetermined term in the
Fefferman-Graham expansions~\cite{FG,GrahamHirachi}.

Reciprocally, if $X\equiv \nu^\sharp(0)$ is a  Killing field on
$\partial M$ and condition \eq{kidscc1} holds, we set
$\alpha=0$ and $\nu^\sharp(x)=\nu^\sharp(0)$. The tensor
$h:=\mathcal L_{\omega^\sharp}g_+$ is then purely tangential
and $h=o(\rho^n)$. Also, as before, $h$ is in the kernel of
$D\mbox{Ein}(g_+)$. We conclude by~\cite{BiquardContinuation}
that $h\equiv 0$ near $\partial M$ so $\omega^\sharp$ is a
Killing for $g_+$.
%
%\erw{trouver aussi des contrexamples si on a pas le bon ordre
%pour h}

\subsection{The Lorentzian case}
 \label{sub8VI0.1}
In the stationary Lorentzian setting, we start by introducing
Gauss coordinates for the space-time metric near $\R\times
\partial M$. All  calculations of Section~\ref{sub10VI0.1} remain valid,
except that now the metric $g$ of \eqref{10VI0.1} is Lorentzian
instead of Riemannian. This does not affect the argument, since
the time-derivatives of all fields involved drop out, and so
the Biquard method leads again to elliptic equations in the
space variables. From there the proof proceeds exactly in the
same fashion, using our unique continuation
Theorem~\ref{maintheorlinear} in place of the Biquard one. This
readily proves Theorem~\ref{maintheorem4}.

The proof of Theorem~\ref{maintheorem5a} proceeds similarly,
using the linearized equivalent of Theorem~\ref{maintheorem1}
(compare Remark~\ref{R12VI0.1}).
\qed

\medskip

\section{Concluding remarks}
 \label{s18VI0.1}

Some readers might be tempted to think that our
finite-boundary-unique-continuation results are a trivial
consequence of the usual analyticity results for stationary
solutions of vacuum Einstein equations~\cite{MzH}, for if the
metric can be extended across the boundary in the stationary
class, then it is analytic at the boundary, and unique
continuation is straightforward. The following example shows
that this is not the case: Recall that Weyl metrics, which are
static axi-symmetric solutions of the vacuum Einstein
equations, are uniquely described by axisymmetric solutions $u$
of the flat-space Laplace equation. So choose some
axi-symmetric simply-connected domain $\Omega$ with analytic
boundary in $\R^3$, and let $u$ be a harmonic function on
$\Omega$ with \emph{non-analytic} boundary values on $\partial
\Omega$. We can choose $\Omega$ and $u$ so that the Killing
vector is uniformly timelike on $\Omega$. Then $u$ cannot be
extended to a harmonic function defined on a set larger than
$\Omega$, otherwise its trace on $\partial \Omega$ would have
been analytic. Thus the corresponding metric $g$ cannot be
extended in the Weyl class.

Now, if $g$ could be extended in the stationary class, then the
associated stationary Killing vector would be timelike in a
neighborhood of $\partial \Omega$, and thus the metric would be
analytic across $\Omega$. By analyticity the extension would be
static and axi-symmetric, and therefore in the Weyl class. But
then $u$ would be analytic on $\partial \Omega$. So  $g$ cannot
be extended across $\partial \Omega$ in the stationary vacuum
class, and an extension within the vacuum class, if any exist,
cannot be analytic at $\partial \Omega$. In particular the
metrics in this example admit no stationary vacuum extensions
away from $\Omega$.  Nevertheless, by our results above, any
KID on the boundary arises from a space-time Killing vector
defined on a one-sided  neighborhood of $\R\times\partial
\Omega$.

\appendix
\section{}
 \label{A12XI0.1}
\begin{lemma}\label{estiA}
$$
 \left||\A_0|^2_{{\mathcal G}_0}-|\A|^2_{\mathcal G}
  \right|
 \leq C (|\A_0-\A|_{{\mathcal G}_0}+|{\mathcal G}_0-{\mathcal G}|_{{\mathcal G}_0}).
$$
\end{lemma}
\begin{proof}
We first recall that, in view of our assumptions on the metric,
we clearly have, for coordinates smooth up to the boundary:
$$ V=O(e^r),\; g=O(e^{2r}),\;
V'-V=O(e^{-r}), \; \zeta'=O(e^{-2r}), \; \II-g=O(1),
$$
in particular, near the boundary,  one has
$$
\II\sim g,\;\A\sim\mathcal G,
$$
and the same is true for the $V_0$, $g_0$, etc.  Let us compute
$$
|\A|_{{\mathcal G}}^2=|\II|^2+(V^{-1}V')^2+\frac12V^{2}|\xi'|^2,
$$
$$
|\A_0|_{{\mathcal G}_0}^2=|\II_0|_0^2+(V_0^{-1}V_0')^2+\frac12V_0^{2}|\xi'_0|_0^2,
$$
$$
|\A-\A_0|_{{\mathcal G}_0}^2=|\II-\II_0|_0^2+V_0^{-4}(VV'-V_0V_0')^2+\frac12V_0^{-2}|V^{2}\xi'-V_0^2\xi_0'|_0^2,
$$
$$
|{\mathcal G}_0-{\mathcal G}|^2_{{\mathcal G}_0}=V_0^{-4}(V^2-V_0^2)^2+|g-g_0|^2_0.
$$
We will show that each term appearing in  $||\A|_{{\mathcal
G}}^2-|\A_0|_{{\mathcal G}_0}^2|$  can be controlled by those
in $|\A-\A_0|_{{\mathcal G}_0}$ and $|{\mathcal G}_0-{\mathcal
G}|_{{\mathcal G}_0}$.

First, formally  (recall $\II\sim g \sim g_0\sim \II_0$) we have
$$
|\II|^2-|\II_0|_0^2=g^{-2}\II^2-g_0^{-2}\II_0^2=(g^{-2}-g_0^{-2})\II^2+g_0^{-2}(\II^2-\II_0^2)\leq C(|g-g_0|_0+|\II-\II_0|_0).
$$
We also have (recall $V'\sim V\sim V_0\sim V_0'$)
\begin{eqnarray*}
|(V^{-1}V')^2-(V_0^{-1}V_0')^2|&\leq &C_1|VV'-V_0V'_0|\\
&\leq&C_1V_0^{-2}|VV'-V_0V_0'+V^{-1}V'(V_0^2-V^2)|\\
&\leq& C(V_0^{-2}|VV'-V_0V_0'|+V_0^{-2}|V_0^2-V^2|).
\end{eqnarray*}
Finally, we can write
$$
V^{2} |\xi'|^2-V_0^{2} |\xi'_0|_0^2=a+b+c,
$$
where
\begin{eqnarray*}
|a|&=&V_0^{-2}||V^2\xi'|_0^2- |V_0^2\xi'_0|_0^2|
\leq
V_0^{-2}|V^2\xi'+ V_0^2\xi'_0|_0|V^2\xi'- V_0^2\xi'_0|_0\\
&\leq& C V_0^{-1}|V^2\xi'- V_0^2\xi'_0|_0
\;,
\\
|b|&=&V_0^{-2}V^4||\xi'|^2- |\xi'|_0^2|
 \leq C_1V^2|\xi'|^2|g-g_0|_0\\
&\leq& C|g-g_0|_0
 \;,\\
|c|&=&V_0^{-2} |V_0^2-V^2|V^2|\xi'|^2
\leq C V_0^{-2} |V_0^2-V^2|
 \;.
\end{eqnarray*}
\end{proof}

\bigskip

\noindent{\sc Acknowledgements}
PTC was supported in part by the Polish Ministry of Science and
Higher Education grant Nr N N201 372736. Useful discussions
with Luc Nguyen are acknowledged.

%
%
%\bibliographystyle{amsplain}
%
%\bibliography{../references/hip_bib,%
%../references/reffile,%
%../references/newbiblio,%
%../references/newbiblio2,%
%../references/bibl,%
%../references/howard,%
%../references/bartnik,%
%../references/myGR,%
%../references/newbib,%
%../references/Energy,%
%../references/erwbiblio,%
%../references/netbiblio}

\def\polhk#1{\setbox0=\hbox{#1}{\ooalign{\hidewidth
  \lower1.5ex\hbox{`}\hidewidth\crcr\unhbox0}}}
  \def\polhk#1{\setbox0=\hbox{#1}{\ooalign{\hidewidth
  \lower1.5ex\hbox{`}\hidewidth\crcr\unhbox0}}} \def\cprime{$'$}
  \def\cprime{$'$} \def\cprime{$'$} \def\cprime{$'$}
\providecommand{\bysame}{\leavevmode\hbox to3em{\hrulefill}\thinspace}
\providecommand{\MR}{\relax\ifhmode\unskip\space\fi MR }
\providecommand{\MRhref}[2]{%
  \href{http://www.ams.org/mathscinet-getitem?mr=#1}{#2}
}
\providecommand{\href}[2]{#2}

\end{document}